\documentclass[twoside, 11pt, reqno]{amsart}

\usepackage[left=2.5cm, right=2.5cm, top=3cm, bottom=2.5cm]{geometry}

\usepackage[colorlinks=true, pdfstartview=FitV, linkcolor=blue,citecolor=blue, urlcolor=blue]{hyperref}

\usepackage[usenames]{xcolor}
\definecolor{labelkey}{rgb}{0,0,1}
\definecolor{Red}{rgb}{0.7,0,0.1}
\definecolor{Green}{rgb}{0,0.7,0}

\usepackage{amsfonts, amssymb, amsmath, amsthm, mathrsfs, bbm, cjhebrew, gensymb, textcomp, mathtools, dsfont,tikz}

\usepackage[normalem]{ulem}

\usepackage{commath, setspace, subcaption, parcolumns, multirow, multicol, accents, comment, marginnote, verbatim, empheq, enumerate, stackrel, enumitem}

\usepackage{cleveref}
\usepackage{graphicx}
\usepackage{epsfig}
\usepackage{psfrag}
\usepackage{float}
\usepackage[active]{srcltx}
\usepackage[pagewise, mathlines]{lineno}%\linenumbers

\newtheorem{Theorem}{Theorem}[section]

\newtheorem{Thm}{Theorem}[section]

\newtheorem{Lem}[Theorem]{Lemma}

\numberwithin{equation}{section}

\newcommand{\al}{\alpha}
\newcommand{\be}{\beta}

\newcommand{\eps}{\epsilon}

\newcommand{\kap}{\kappa}
\newcommand{\lam}{\lambda}
\newcommand{\Lam}{\Lambda}
\newcommand{\si}{\sigma}
\newcommand{\xe}{\xi-\eta}

\newcommand{\tht}{\theta}

\newcommand{\lpj}{\triangle_j}

\newcommand{\ZZ}{\mathbb{Z}}
\newcommand{\RR}{\mathbb{R}}

\newcommand{\lb}{\big\langle}
\newcommand{\rb}{\big\rangle}
\newcommand{\lbn}{\langle}
\newcommand{\rbn}{\rangle}

\newcommand{\Sob}[2]{\lVert#1\rVert_{#2}}
\newcommand{\nrm}[1]{\lVert#1\rVert}

\newcommand{\bdy}{\partial}

\newcommand{\Hdot}{\dot{H}}

\newcommand{\Acal}{\mathcal{A}}
\newcommand{\Bcal}{\mathcal{B}}

\DeclareMathOperator{\supp}{supp}

\title[Asymptotic study of supercritical generalized SQG equations in critical Sobolev spaces]{Asymptotic study of supercritical generalized SQG equations in critical Sobolev spaces}
\author{Anuj Kumar}
\thanks{}
\begin{document}
\begin{abstract}
	We study the long time behavior of regular solutions of the supercritical gSQG equations in the fully nonlinear regime. More precisely, under the assumption of small initial data in the critical Sobolev norm, we prove the existence of the unique global solution that satisfies the energy inequality \eqref{energy:inequality} and for which the critical norm $\Sob{\tht(t)}{H^{1+\be-2\al}}$ decays to 0 as time goes to infinity.
\end{abstract}
\maketitle
{\noindent \small {\it {\bf Keywords: generalized surface quasi-geostrophic (gSQG)
    equation, Smooth solution, Global regularity, Energy inequality}
 }  }

%==========================================Introduction================================
\section{Introduction}
In this paper, we study the asymptotic behavior of solutions for the following generalized surface quasi-geostrophic (gSQG) equation:
\begin{align}\label{gSQG}
    \begin{cases}
        &\bdy_t \tht+(-\Delta)^{\al}\tht+u_{\tht}\cdot \nabla \tht=0,\\
        &u_{\tht}=\left(\bdy_{x_2}\Lam^{\beta-2}\tht,-\bdy_{x_1}\Lam^{\beta-2}\tht\right),\\
        &\tht(x,0)=\tht_0(x),\quad x\in \mathbb{R}^2,\,t>0,
    \end{cases}
\end{align}
where $\tht(x,t)$ denotes the evolving scalar, $u_\tht(x,t)$ denotes the advecting velocity field, $\Lam$ denotes the fractional Laplacian $(-\Delta)^{1/2}$. $\al \in (0,1)$ denotes the dissipation parameter and $\beta\in (0,2)$ denotes the constitutive law parameter. The domain is assumed to be $\mathbb{R}^2$. 
%In this regime, \eqref{gSQG} represents a fully nonlinear evolution of the scalar $\tht$.
\par The family of equations in \eqref{gSQG} was first introduced in \cite{ChaeConstantinWu2012} and \cite{ChaeConstantinCordobaGancedoWu2012}, while its inviscid counterpart was studied in \cite{ChaeConstantinWu2011}.  For $\be \in [0,1]$, \eqref{gSQG} represents a family of active scalar equations that interpolate between the 2D Euler equation in vorticity form ($\be=0$) and the dissipative surface quasi-geostrophic equation (SQG) ($\be=1$). On the other hand, for $\be \in (1,2)$, \eqref{gSQG} represents a family with more singular constitutive laws than the SQG equation. The SQG equation is a fundamental equation in geophysics and represents a model for the evolution of temperature or buoyancy of a strongly stratified fluid in a rapidly rotating regime.  It has attracted a lot of interest due to the presence of mechanisms exhibiting vortex-stretching analogous to the 3D Euler equation \cite{ConstantinMajdaTabak1994,Cordoba1998} and the presence of turbulent characteristics analogous to the 2D Euler equation \cite{MajdaTabak1996}. It is also an important toy model to study questions of well-posedness, being a 2D equation for which global regularity remains an outstanding open problem in the supercritical regime, even though global regularity has been established in the subcritical \cite {ConstantinWu1999,Resnick1995} and critical regimes  \cite{CaffarelliVasseur2010,ConstantinVicol2012,KiselevNazarovVolberg2007}. In the supercritical regime, results on local well-posedness for large data, global well-posedness for small data, and Gevrey class smoothing of solutions in time are available \cite{Miura2006, ChenMiaoZhang2007,HmidiKeraani2007, Wu2004, Biswas2014, BiswasMartinezSilva2015, Dong2010}.
\par  
In the case of the gSQG equations \eqref{gSQG}, one usually distinguishes between the subcritical ($\be<2\al$), critical ($\be=2\al$), and supercritical ($\beta>2\al$) regimes. Heuristically, these regimes correspond to the level of compensation of the loss of derivatives in the nonlinear term with the gain imposed by the dissipative term. 
Results establishing global-in-time regularity of solutions are available in subcritical and critical regimes \cite{ChaeConstantinWu2011,ChaeConstantinWu2012,MiaoXue2012} while
in the supercritical regime, the question of global-in-time regularity remains open. Nevertheless, several results proving local-in-time regularity for large initial data and global-in-time regularity for small initial data have been established \cite{HuKukavicaZiane2015,ChaeConstantinCordobaGancedoWu2012, JollyKumarMartinez2020a}.
\par In the more singular regime of the gSQG equations, i.e, when $\be\in (1,2)$, one distinguishes between two regimes based on the structure of the equations: $\be \in (1,2\al+1)$ in which case \eqref{gSQG} defines a quasilinear evolution of the scalar $\tht$ and $\be \in [2\al+1,2)$ in which case \eqref{gSQG} represents a fully nonlinear equation. In this paper, we study the asymptotic behavior of solutions in the regime where \eqref{gSQG} defines a fully nonlinear equation.
\par
The equations in \eqref{gSQG} are invariant under the following scaling:
\[\tht_{\lambda}(t,x)=\lambda^{2\al-\be}\tht(\lambda^{2\al}t,\lambda x)\]
Specifically, if $\tht$ is a solution of \eqref{gSQG} with initial data $\tht_0$, then $\tht_{\lambda}$ is also a solution of \eqref{gSQG} with initial data $(\tht_0)_{\lambda}$. It is straightforward to check that $\Hdot^{1+\be-2\al}$ is a critical space for \eqref{gSQG}.
For $\al \in(0,1/2)$ and $\be \in (1,2)$, the local existence and uniqueness for large initial data and global existence and uniqueness for small initial data were established in \cite{JollyKumarMartinez2020a}. We recall this result below in \cref{thm:main:beta:recall}. 
\begin{Thm}[\cite{JollyKumarMartinez2020a}]\label{thm:main:beta:recall}
		Let $\be\in(1,2)$  and $\al\in(0,1/2)$. For each $\tht_0\in H^{1+\be-2\al}(\RR^2)$, there exists $T>0$ and a unique solution, $\tht$, of \eqref{gSQG} such that 
		    \begin{align*}
		        \tht\in C([0,T);H^{1+\be-2\al})\cap L^2(0,T;\Hdot^{1+\be-\al}).
		    \end{align*}
If $\nrm{\tht_0}_{\Hdot^{1+\be-2\al}}$ is small enough, then $T=\infty$ is allowed.
	\end{Thm}
    
Our main result is stated in \cref{thm:main:beta} consisting of three parts, first, we prove a regularity criterion that the local solution obtained in \cref{thm:main:beta:recall} must satisfy up to the maximal time of existence $T^*$. Second, we establish the existence of the unique global solution under the assumption of small $\Sob{\tht_0}{H^{1+\be-2\al}}$ which satisfies the energy inequality \eqref{energy:inequality}. Third, we establish that the critical norm $\Sob{\tht(t)}{H^{1+\be-2\al}}$ decays to 0 as time goes to infinity. 
\begin{Thm}\label{thm:main:beta}
		Assume that $\al\in(0,1/2)$, $\be \in [2\al+1,2)$ and $\tht_0\in H^{1+\be-2\al}(\RR^2)$. Let $\tht\in C([0,T^*);H^{1+\be-2\al})$ denote the maximal solution for \eqref{gSQG} obtained in \cref{thm:main:beta:recall}. Then, the following statements hold:
        \begin{itemize}
            \item[i)] 
            \begin{align}\label{regularity:criterion}
            T^*<\infty \implies \int_0^{T^*}\Sob{\Lam^\al \tht(s)}{H^{1+\be-2\al}}^2=\infty;\end{align}
            \item[ii)] There is a constant $C(\al,\be)>0$ such that $\Sob{\tht_0}{H^{1+\be-2\al}}<C(\al,\be)$ implies the existence of a unique global solution, $\tht$ of \eqref{gSQG} such that
            \begin{align}\label{energy:inequality}
               & \tht\in C([0,T);H^{1+\be-2\al})\cap L^2(0,T;\Hdot^{1+\be-\al}),\notag\\
                &\Sob{\tht(t)}{H^{1+\be-2\al}}^2+\int_{0}^{t}\Sob{\Lam^{\al}\tht}{H^{1+\be-2\al}}^2\,ds\le \Sob{\tht_0}{H^{1+\be-2\al}}^2,\quad \forall t \in [0,\infty);
            \end{align}
            \item[iii)] Let $\tht$ be the unique global solution obtained in ii). Then,
            \begin{align}
                \lim_{t\to \infty}\Sob{\tht(t)}{H^{1+\be-2\al}}=0.
            \end{align}
        \end{itemize}
	\end{Thm}
The corresponding results in the quasilinear regime, i.e. when $\be \in (1,2\al+1)$,  were established recently in \cite{Melo2024} while the case of SQG ($\be=1$) was done in \cite{BenameurKatar2022} (see also \cite{BenameurBlel2014, BenameurBenhamed2015, BenameurAbdallah2021,AmaraBenameur2022}). Additionally, when $\be\in (2\al,\al+1)$ and $\tht_0\in \Hdot^{s}\cap L^p \cap \dot{B}^{\be-2\al}_{\infty,\infty}$ is assumed to be small in $\dot{B}^{\be-2\al}_{\infty,\infty}$, ($s>0$ and $p\in[1,2)$), it was shown in \cite{Ye2022} that $\Sob{\tht(t)}{H^s}\to 0$ as $t\to \infty$. Therefore, \cref{thm:main:beta} extends the result of \cite{Melo2024} and \cite{Ye2022} to the fully nonlinear regime of the gSQG equations. To prove \cref{thm:main:beta}, we follow an approach that combines the methods used in \cite{Melo2024} and \cite{BenameurKatar2022} along with Littlewood-Paley decomposition techniques that enable us a finer control over the nonlinear interactions in the energy arguments. Along with sharp commutator estimates established previously in \cite{JollyKumarMartinez2020a, JollyKumarMartinez2020b}, we prove a preliminary lemma stated in \cref{lem:main}. This lemma is the counterpart of lemma 2.3 in \cite{Melo2024} and provided the basis for extending the methods of \cite{BenameurKatar2022} beyond the SQG equations. 
\section{Mathematical Preliminaries}
In this section, we present the relevant notations and definitions used in the forthcoming sections.
Let $\mathscr{S}(\RR^2)$ be the space of Schwartz class functions on $\RR^2$ and $\mathscr{S}'(\RR^2)$ denote the space of tempered distributions. We denote by $\hat{f}$ or $\mathcal{F}(f)$, the Fourier transform of $f$, which is defined by
	\[\hat{f}(\xi)\overset{}{:=}\int e^{-2\pi i x\cdot \xi}f(x)dx,\quad f\in\mathscr{S}'(\RR^2).\] 
   We recall that for $f,g$, we have
	    \begin{align}\notag
	        \lbn f,g \rbn_{L^2}=\lbn\hat{f},\hat{g}\rbn_{L^2}.
	    \end{align}
   $\Lam^\si$ denotes the fractional laplacian operator which is defined as
	    \begin{align}\notag
	        \mathcal{F}(\Lam^\si f)(\xi)=|\xi|^\si\mathcal{F}(f),\quad \si\in\RR.
	    \end{align} 
     Next, we define the $L^2$-based homogeneous and inhomogeneous Sobolev spaces. For $\si \in \mathbb{R}$, we have
	    \begin{align}
	        &\Hdot^\si(\mathbb{R}^2){:=}\left\{f\in \mathscr{S}(\RR^2):\hat{f}\in L^2_{loc},\, \Sob{f}{\Hdot^\si}:=\Sob{\Lam^\si f}{L^2}<\infty\right\},\label{def:hom:Sob:norm}\\
	        &H^\si(\mathbb{R}^2){:=}\left\{f\in \mathscr{S}(\RR^2):\hat{f}\in L^2_{loc}, \Sob{f}{H^\si}{=}\Sob{(I-\Delta)^{\si/2}) f}{L^2}<\infty\right\}.\label{def:inhom:Sob:norm}
	    \end{align}
Since for $\si \ge 0$, we have
\[2^{-\si}\{1+|\xi|^2\}^{\si}\le 1+|\xi|^{2\si}\le 2^\si\{1+|\xi|^2\}^\si,\quad \forall \xi\in \mathbb{R}^2.\]
Hence, we use the following equivalent definition of the inhomogeneous Sobolev spaces $H^{\si}$ when $\si\ge0$:
\begin{align}
    H^\si(\mathbb{R}^2){:=}\left\{f\in \mathscr{S}(\RR^2):\hat{f}\in L^2_{loc}, \Sob{f}{H^\si}^2{:=}\Sob{f}{L^2}^2+\Sob{f}{\Hdot^{\si}}^2<\infty\right\}.\label{def:inhom:Sob:norm}
\end{align}
Next, we give a brief introduction to Littlewood-Paley decomposition of functions.
We define
	\begin{align*}
	\mathscr{Q}(\RR^2){:=}\left\{f\in \mathscr{S}(\RR^2): \int f(x)x^{\tau}\, dx=0, \quad \abs{\tau}=0,1,2,\cdots \right\}.
	\end{align*}
	Let us denote by $\mathscr{Q}(\RR^2)'$, the topological dual of $\mathscr{Q}(\RR^2)$. Then, $\mathscr{Q}(\RR^2)'$ can be identified with the space of tempered distributions modulo the vector space of polynomials on $\mathbb{R}^2$, denoted by $\mathscr{P}$, i.e.
	\begin{align*}
	\mathscr{Q}'(\RR^2)\cong\mathscr{S}(\RR^2)/\mathscr{P}.
	\end{align*}
	Let ${\Bcal}(r)$ denote the open ball centered at the origin with radius $r$ and ${\Acal}(r_{1},r_{2})$ denote the open annulus centered at the origin with inner and outer radii $r_{1}$ and $r_{2}$. We can find two non-negative radial functions denoted $\chi,\phi\in\mathscr{S}(\RR^2)$ such that $\supp\chi\subset{\Bcal}(1)$ and $\supp\phi\subset{\Acal}(2^{-1},2)$ and such that they satisfy the following conditions 
	\begin{align}
	    \begin{cases}
	    \sum_{j\in\ZZ}\phi(2^{-j}\xi)=1,\\
	    \chi(\xi)+\sum_{j\geq0}\phi(2^{-j}\xi)\equiv 1,\,\forall \xi\in\RR^2\setminus\{\mathbf{0}\},
	    \end{cases}\label{LP:conditions}
	\end{align}
    We denote by
    \[\phi_j(\xi){=}\phi(2^{-j}\xi),\quad \chi_j(\xi){=}\chi(2^{-j}\xi)\]
    and
    \begin{align}\notag            {\Acal}_{j}= {\Acal}(2^{j-1},2^{j+1}),\quad{\Acal}_{\ell,k}= {\Acal}(2^{\ell},2^{k}),\quad  {\Bcal}_j={\Bcal}(2^j).
        \end{align}
    Note that
	    \begin{align}\label{rewrite:supp}
	        \supp\phi_j\subset{\Acal}_j,\quad \supp\chi_j\subset{\Bcal}_j.
	    \end{align}
    It is straightforward to check that the following almost-orthogonality conditions hold
    \[\supp\phi_i\cap\supp\phi_j=\varnothing,\,\text{if}\,
	    |i-j|\geq2,\quad
	    \text{and}\quad\supp\phi_i\cap\supp\chi =\varnothing, \,\text{if}\,i\ge 1.\]
	We denote by ${\lpj}$ and $S_{j}$, the (homogeneous) Littlewood-Paley dyadic blocks defined by
	\begin{align}\notag
	\mathcal{F}({\lpj}f)=\phi_{j}\mathcal{F}(f), \quad \mathcal{F}(S_{j}f)=\chi_{j}\mathcal{F}(f). 
	\end{align}
    Clearly, by \eqref{LP:conditions}, we have
	\begin{align*}
	&\mathcal{F}({\lpj}f)|_{{\Acal}_j^c}=0,\quad
	\mathcal{F}(S_{j}f)|_{{\Bcal}_j^{c}}=0,
	\end{align*}
    Note that for any $f \in \mathscr{S}(\RR^2)$, we have
    \begin{align*}
         f&=S_if+\sum_{j\geq i}\lpj f,\quad i\in\ZZ.
    \end{align*}
	and for that any $f\in\mathscr{Q}(\RR^2)'$, we have
	    \begin{align*}
	        f&=\sum_{j\in\ZZ}\lpj f.
	    \end{align*}
We have the following equivalence
	    \begin{align}\label{norm:equivalence}
	       C^{-1}\Sob{f}{L^2}\leq \left(\sum_{j\in \mathbb{Z}}\left(\nrm{\lpj f}_{L^2}\right)^{2}\right)^{\frac{1}{2}}\leq C\Sob{f}{L^2},
	    \end{align}
	for some constant $C$. We recall the following classical inequality (see \cite{BahouriCheminDanchinBook2011}):
    \begin{Lem}[Bernstein inequalities]\label{T:Bernstein}
		Let $\si\in\RR$ and $1\le p \le q\le \infty$. Then
		\begin{align*}
	C^{-1}2^{\si j}\nrm{{\lpj}f}_{L^q(\mathbb{R}^2)}\le \nrm{\Lam^{\si}{\lpj}f}_{L^q(\mathbb{R}^2)}\le C 2^{\si j+2j(\frac{1}{p}-\frac{1}{q})}\nrm{{\lpj}f}_{L^p(\mathbb{R}^2)},
		\end{align*}
	 where $C>0$ is a constant that depends on $p,q$ and $\si$.
	\end{Lem}
     We denote the commutator of two operators, $S$ and $T$, by $[S,T]$, where
    \[[S,T]:=ST-TS.\]
    We require the following commutator estimate which was proved in \cite{JollyKumarMartinez2020a}:
\begin{Lem}[\cite{JollyKumarMartinez2020a}]\label[Lem]{lem:commutator1}
Let $s\in(0,1)$,  ${\rho_1\in (0,2)}$, ${\rho_2} \in (-1,1)$ be such that $\rho_2>\rho_{1}-1$. Let $f\in\dot{H}^{\rho_2}(\mathbb{R}^2)$ and $g \in \dot{H}^{2-s-\rho_1}(\mathbb{R}^2)$. Then there exists a constant $C>0$, depending only on $s,{\rho_1,\rho_2}$, such that
    \begin{align}
        \Sob{[\Lam^{-s}\bdy_\ell,g]f}{\Hdot^{\rho_2-\rho_1}}\leq C\Sob{g}{\dot{H}^{2-s-\rho_1}}\Sob{ f}{\Hdot^{\rho_2}},\quad \ell=1,2.\notag
    \end{align}
\end{Lem}
Next, we recall the following commutator estimate for operators which are of the form of a product of Fourier multiplier operators given by $\Lam^{\si}, \bdy_{\ell}$, and $\triangle_j$. 
We remark that Lemma \ref{lem:commutator2} for the setting of $i=j$ was proved in \cite{JollyKumarMartinez2020b}. The proof in the current setting of $i\sim j$ is similar. For the convenience of the reader, we provide a brief sketch of the proof in Appendix \ref{sec:app}. 
\begin{Lem}[\cite{JollyKumarMartinez2020b}]\label[Lem]{lem:commutator2}
Let $s \in [0,1)$, $\rho \in \mathbb{R}$, and $\nu \in (0,2)$ be such that $s>\nu-1$. Given $i, j\in \mathbb{Z}$ such that $|i-j|\le k$ for a positive integer $k$, suppose that $\supp\hat{h}\subset{\Acal}_j$. 
Then there exists a sequence $\{c_j\}\in\ell^2(\ZZ)$ satisfying $\Sob{\{c_j\}}{\ell^2}\leq1$ and such that
        \begin{align}\label{commutator2:A}
        |\lb [\Lam^{s+\rho}\bdy_{\ell}\Delta_i,g]f,h  \rb|
        \leq Cc_{j}
              \min\left\{ \Sob{{f}}{\Hdot^{1-\nu}}\Sob{g}{\Hdot^{s+1}},\Sob{{f}}{\Hdot^s}\Sob{g}{\Hdot^{2-\nu}}\right\}\Sob{h}{\Hdot^{\rho+\nu}}
    \end{align}
    and
     \begin{align}\label{commutator2:B}
        |\lb [\Lam^{s+\rho+1}\Delta_i,g]f,h  \rb|
        \leq Cc_{j}
              \min\left\{ \Sob{{f}}{\Hdot^{1-\nu}}\Sob{g}{\Hdot^{s+1}},\Sob{{f}}{\Hdot^s}\Sob{g}{\Hdot^{2-\nu}}\right\}\Sob{h}{\Hdot^{\rho+\nu}}
    \end{align}
hold for some constant $C>0$, depending only on $s,\rho, \nu$. 
\end{Lem}
We now state the main lemma attributed to earlier. Unlike in \cite{Melo2024}, in the fully nonlinear case of $\be \in [2\al+1,2)$, we establish the control over nonlinear terms as they appear in energy arguments upon action by Littlewood-Paley dyadic operators. The proof involves a finer control over nonlinear interactions afforded by decomposing the terms in the frequency space and applying commutator estimates \cref{lem:commutator1} and \cref{lem:commutator2}.
\begin{Lem}\label[Lem]{lem:main}
    Let $\al \in (0,1)$, $\be\in [2\alpha+1,2)$, and $\si \in (2-\al,3-\al)$. For $\tht\in H^{1+\be-2\al}(\mathbb{R}^2)\cap \dot{H}^{\si+\al}(\mathbb{R}^2)$, we have the following inequality
    \begin{align}\label{lem1}
    \sum_{j\in \mathbb{Z}}|\lbn \Delta_j (u_{\tht}\cdot \nabla\tht),\Delta_j(-\Delta)^{\si}\tht\rbn_{L^2}|\le C\Sob{\tht}{\dot{H}^{1+\be-2\al}}\Sob{\tht}{\dot{H}^{\si+\al}}^2,
    \end{align}
    where $C=C(\al, \be, \si)$ is a positive constant.
\end{Lem}
\begin{proof}
For convenience, we define the following notations
    \begin{align}\label{def:Lamj:notation}
         \Lam_i^s:=\Delta_i \Lam^s=\Lam^s \Delta_i.
    \end{align}
%    We treat the two cases $\si \in [1,2-\al)$ and $\si \in [2-\al,3-\al)$ separately.
%    Since $\nabla \cdot u_\tht =0$, we have for all $s\ge 0$,
%    \begin{align}\label{zero:term}
 %       \lbn u_{\tht}\cdot \nabla\Lam^s_j\tht,\Lam^s_j\tht\rbn_{L^{2}}=0.
%    \end{align}
%\subsubsection*{Case: $\si \in [1,2-\al)$} 
 %   Using \eqref{zero:term}, we have
 %   \begin{align*}
   %     \lbn \Delta_j (u_{\tht}\cdot \nabla\tht),\Delta_j(-\Delta)^{\si}\tht\rbn_{L^2}&=\lbn \Lam^\si_j (u_{\tht}\cdot \nabla\tht),\Lam^\si_j\tht\rbn_{L^2}-\lbn u_{\tht}\cdot \nabla \Lam^\si_j\tht,\Lam^\si_j\tht\rbn_{L^2}\\
 %       &=\lbn [\Lam^\si_j, u_\tht]\cdot \nabla \tht, \Lam^\si_j \tht \rbn=:I.
  %  \end{align*}
%Applying \cref{lem:commutator2} with $s=\si-1+\al,\, \rho=-\al,\, \nu=2\al$, we obtain
%\begin{align*}
%    |I|&\le C_j\Sob{u_\tht}{\Hdot^{1-2\al}}\Sob{\nabla \tht}{\Hdot^{\si-1+\al}}\Sob{\Lam_j^\si \tht}{\Hdot^{\al}}\\
%    &\le C_j\Sob{\tht}{\Hdot^{\be+1-2\al}}\Sob{\tht}{\Hdot^{\si+\al}}\Sob{\Lam^{\si+\al}_j \tht}{L^2}
%\end{align*}
%for $C_j \in \ell^2(\mathbb{Z})$.
%Upon summing over $j\in \mathbb{Z}$ and applying the Cauchy-Schwarz %inequality, we obtain the claimed estimate \eqref{lem1}.
%\subsubsection*{Case: $\si \in [2-\al,3-\al)$} 
Since \[\nabla^{\perp}\tht\cdot \nabla \tht=0,\]
we have
\begin{align*}
    \lbn \Delta_j (u_{\tht}\cdot \nabla\tht),\Delta_j (-\Delta)^{\si}\tht\rbn_{L^{2}}&=\left\{\lbn \Lam_j^{\si} (u_{\tht}\cdot \nabla\tht),\Lam_j^{\si}\tht\rbn_{L^{2}}-\lbn\Lam_j^{\si+\be-2}(\nabla^{\perp}\tht\cdot \nabla \tht),\Lam_j^{\si}\tht\rbn_{L^2}\right\}\\
    &=J_1+J_2+J_3,
\end{align*}
where 
\begin{align*}
    J_1&=\left\{\lbn  \nabla^{\perp} \Lam_j^{\si+\be-2}  \tht\cdot \nabla\tht,\Lam_j^{\si}\tht\rbn_{L^{2}}-\lbn  \nabla^{\perp}\Lam^{\be-2}\cdot(\Lam_j^{\si}{\tht} \nabla\tht),\Lam_j^{\si}\tht\rbn_{L^{2}}\right\},
\end{align*}
\begin{align*}
    J_2&=\left\{\lbn \Lam_j^{\si} (u_{\tht}\cdot \nabla\tht),\Lam_j^{\si}\tht\rbn_{L^{2}}-\lbn  \nabla^{\perp} \Lam_j^{\si+\be-2}  \tht\cdot \nabla\tht,\Lam_j^{\si}\tht\rbn_{L^{2}}\right\},
\end{align*}
and 
\begin{align*}
    J_3&=\left\{-\lbn\Lam_j^{\si+\be-2}(\nabla^{\perp}\tht\cdot \nabla \tht),\Lam_j^{\si}\tht\rbn_{L^2}+\lbn  \nabla^{\perp}\Lam^{\be-2}\cdot(\Lam_j^{\si}{\tht} \nabla\tht),\Lam_j^{\si}\tht\rbn_{L^{2}}\right\}.
\end{align*}
 Since $\si \ge 2-\alpha$, we have for any function $\psi$
\begin{align}\label{laplacian}
    \Lam^{\si} \psi=\Lam^{\si-2}(-\Delta)\psi=-(\Lam^{\si-2}\bdy_l)\bdy_l \psi
\end{align}
Since $\lbn u_\tht \cdot \nabla \Lam^\si_j \tht, \Lam^\si_j \tht\rbn=0$, we can express $J_2$ as
\[J_2=\lbn \Lam_j^{\si} (u_{\tht}\cdot \nabla\tht),\Lam_j^{\si}\tht\rbn_{L^{2}}-\lbn  \nabla^{\perp} \Lam_j^{\si+\be-2}  \tht\cdot \nabla\tht,\Lam_j^{\si}\tht\rbn_{L^{2}}-\underbrace{\lbn u_{\tht}\cdot \nabla\Lam^\si_j\tht,\Lam^\si_j\tht\rbn_{L^{2}}}_{=0}\]
Applying \eqref{laplacian} and the product rule, we obtain
\begin{align}\label{productrule:split}
    J_2&=-\lbn \Lam^{\si-2}_j\bdy_{l}(\nabla^{\perp}\Lam^{\be-2}\bdy_{l}\tht\cdot \nabla \tht) ,\Lam^{\si}_j\tht\rbn_{L^{2}}+\lbn (\nabla^{\perp}\Lam^{\be-2}\Lam^{\si-2}\bdy_l \bdy_l\tht \cdot\nabla \tht),\Lam^{\si}_j \tht \rbn_{L^{2}} \notag\\
    &-\lbn \Lam^{\si-2}_j\bdy_l(\nabla^{\perp}\Lam^{\be-2}\tht\cdot \nabla\bdy_l \tht),\Lam^{\si}_j \tht\rbn_{L^{2}}+\lbn (\nabla^{\perp}\Lam^{\be-2}\tht\cdot \Lam^{\si-2}_j\bdy_l \bdy_l \tht),\Lam^{\si}_j\rbn_{L^{2}} \notag \\
    &=-\lbn [\Lam^{\si-2}_j \bdy_l,\bdy_\ell]\bdy_\ell^{\perp}\bdy_l \Lam^{\be-2}\tht,\Lam^{\si}_j \tht\rbn_{L^{2}}-\lbn [\Lam^{\si-2}_j \bdy_l,\bdy_\ell^{\perp}\Lam^{\be-2}\tht]\bdy_{\ell}\bdy_l \tht,\Lam^{\si}_j\rbn_{L^{2}}\notag\\
    &=:\, J_{21}+J_{22}.
\end{align}

Similarly, using the fact that $\lbn  (\nabla^{\perp}\tht\cdot\nabla\Lam^{\si+\frac{\be-2}{2}}_j\tht,\Lam^{\si+\frac{\be-2}{2}}_j\tht\rbn_{L^{2}}=0$, $J_3$ can be expressed as 
\[J_3=-\lbn\Lam_j^{\si+\be-2}(\nabla^{\perp}\tht\cdot \nabla \tht),\Lam_j^{\si}\tht\rbn_{L^2}+\lbn  \nabla^{\perp}\Lam^{\be-2}\cdot(\Lam_j^{\si}{\tht} \nabla\tht),\Lam_j^{\si}\tht\rbn_{L^{2}}+\underbrace{\lbn  (\nabla^{\perp}\tht\cdot\nabla\Lam^{\si+\frac{\be-2}{2}}_j\tht,\Lam^{\si+\frac{\be-2}{2}}_j\tht\rbn_{L^{2}}}_{=0}\]
Similar to \eqref{productrule:split}, applying \eqref{laplacian} and the product rule, we obtain
\begin{align}
    J_3&=-\lbn [\Lam^{\si-2}_j\bdy_l,\bdy_\ell \tht]\bdy_\ell^{\perp}\bdy_l \tht,\Lam^{\si+\be-2}_j \tht \rbn_{L^{2}}-\lbn [\Lam_j^{\si-2+\frac{\be-2}{2}}\bdy_l,\bdy_\ell^{\perp}\tht]\bdy_\ell\bdy_l\tht,\Lam^{\si+\frac{\be-2}{2}}_j \tht\rbn_{L^{2}}\notag\\
    &=:\,J_{31}+J_{32}.
\end{align}
We write
\begin{align*}
J_1=-\lbn [\bdy_{\ell}^{\perp}\Lam^{\be-2}, \bdy_{\ell}\tht]\Lam^{\si}_{j}\tht,\Lam^{\si}_{j}\tht \rbn
\end{align*}
and apply \cref{lem:commutator1} with $s=2-\be$, $\rho_1=2\al$ and $\rho_2=\al$, to obtain
\begin{align}
    |J_1|&\le C\Sob{[\bdy_{\ell}^{\perp}\Lam^{\be-2}, \bdy_{\ell}\tht]\Lam^{\si}_{j}\tht}{\Hdot^{-\al}}\Sob{\Lam^\si_j}{\Hdot^\al} \notag\\&\le C\Sob{\bdy_\ell \tht}{\Hdot^{\be-2\al}}\Sob{\Lam^\si_j}{\Hdot^\al}^2.
\end{align}
Summing over $j$ and using \eqref{norm:equivalence}, we obtain 
\begin{align}\label{J1:final:estimate}
    \sum_{j\in \mathbb{Z}} |J_1|\le C\Sob{\tht}{\dot{H}^{1+\be-2\al}}\Sob{\tht}{\dot{H}^{\si+\al}}^2.
\end{align}
Applying \cref{lem:commutator2} with $s=\si-2+\al$, $\rho=-\al$, and $\nu=2\al$, we obtain 
\begin{align}\label{J21}
    |J_{21}|&\le Cc_j\Sob{\bdy_\ell^{\perp}\bdy_l\Lam^{\be-2}\tht}{\Hdot^{1-2\al}}\Sob{\bdy_\ell \tht}{\Hdot^{\si-1+\al}}\Sob{\Lam_j^\si \tht}{\Hdot^\al}\notag \\
    &\le Cc_j\Sob{\tht}{\Hdot^{\be+1-2\al}}\Sob{\tht}{\Hdot^{\si+\al}}\Sob{\Lam_j^\si\tht}{\Hdot^{\al}}.
\end{align}
Applying \cref{lem:commutator2} with $s=\si-2+\al$, $\rho=-\al$, and $\nu=2\al$, we obtain 
\begin{align}\label{J22}
    |J_{22}|&\le Cc_j\Sob{\bdy_\ell \bdy_l \tht}{\Hdot^{\si-2+\al}}\Sob{\bdy_\ell\Lam^{\be-2}\tht}{\Hdot^{2-2\al}}\Sob{\Lam^{\si}_j\tht}{\Hdot^\al}\notag\\
    &\le \Sob{\tht}{\Hdot^{\si+\al}}\Sob{\tht}{\Hdot^{\be+1-2\al}}\Sob{\Lam_j^\si\tht}{\Hdot^{\al}}.
\end{align}
Applying \cref{lem:commutator2} with $s=\si-2+\al$, $\rho=-\al$, and $\nu=2-\be+2\al$, we obtain 
\begin{align}\label{J31}
    |J_{31}|&\le Cc_j \Sob{\bdy_\ell^{\perp}\bdy_l \tht}{\Hdot^{\be-1-2\al}}\Sob{\bdy_\ell \tht}{\Hdot^{\si-2+\al+1}}\Sob{\Lam^{\si+\be-2}_j \tht}{\Hdot^{2-\be+\al}}\notag\\
    &\le Cc_j \Sob{\tht}{\Hdot^{\be+1-2\al}}\Sob{\tht}{\Hdot^{\si+\al}}\Sob{\Lam^{\si+\be-2}_j \tht}{\Hdot^{2-\be+\al}}.
\end{align}
Applying \cref{lem:commutator2} with $s=\si-2+\al$, $\rho=\frac{\be-2}{2}-\al$, and $\nu=2-\be+2\al$, we obtain 
\begin{align}\label{J32}
    |J_{32}|&\le Cc_j \Sob{\bdy_\ell \bdy_l \tht}{\Hdot^{\si-2+\al}}\Sob{\bdy_\ell^{\perp}\tht}{\Hdot^{\be-2\al}}\Sob{\Lam^{\si+\frac{\be-2}{2}}\tht}{\Hdot^{\frac{2-\be}{2}+\al}}\notag\\
    &\le Cc_j \Sob{\tht}{\Hdot^{\si+\al}}\Sob{\tht}{\Hdot^{\be+1-2\al}}\Sob{\Lam^{\si+\frac{\be-2}{2}}\tht}{\Hdot^{\frac{2-\be}{2}+\al}}.
\end{align}

Finally, we sum in $j\in \mathbb{Z}$, apply the Cauchy-Schwarz inequality in \eqref{J21}-\eqref{J32}, and use \eqref{norm:equivalence} to obtain
\begin{align}\label{J2J3:final:estimate}
    \sum_{j\in \mathbb{Z}}\left(|J_{21}|+|J_{22}|+|J_{31}|+|J_{32}|\right)\le C\Sob{\tht}{\dot{H}^{1+\be-2\al}}\Sob{\tht}{\dot{H}^{\si+\al}}^2.
\end{align}
Collecting the estimates \eqref{J1:final:estimate} and \eqref{J2J3:final:estimate} completes the proof.
\end{proof}
The following Lemma follows directly from \cref{lem:main} by taking $\si=1+\be-2\al$.
%We require the particular case of Lemma 2.1 obtained by taking $\si=1+\be-2\al$.
\begin{Lem}\label[Lem]{lem:main:particular}
    Let $\al \in (0,1)$, $\be\in [2\alpha+1,2)$, and $\si \in (2-\al,3-\al)$. For $\tht\in \Hdot^{1+\be-2\al}(\mathbb{R}^2)\cap \dot{H}^{1+\be-\al}(\mathbb{R}^2)$, we have the following inequality
    \begin{align}\label{lem2}
    \sum_{j\in \mathbb{Z}}|\lbn \Delta_j (u_{\tht}\cdot \nabla\tht),\Delta_j(-\Delta)^{1+\be-2\al}\tht\rbn_{L^2}|\le C\Sob{\tht}{\dot{H}^{1+\be-2\al}}\Sob{\tht}{\dot{H}^{1+\be-\al}}^2,
    \end{align}
    where $C=C(\al, \be)$ is a positive constant.
\end{Lem}
We recall the following classical result from \cite{Brezisbook2011}:
\begin{Lem}[\cite{Brezisbook2011}]\label[Lem]{lem:Hilbert:conv}
    Let $H$ be a Hilbert space and $\{x_n\}_{n\in \mathbb{N}}$ be a bounded sequence in $H$ such that
    \begin{align*}
        x_n \rightharpoonup x,\quad \text{and}\,\quad \limsup_{n\to \infty}\Sob{x_n}{}\le \Sob{x}{},
    \end{align*}
then
\[\lim_{n\to \infty}\Sob{x_n-x}{}=0.\]
\end{Lem}

\section{Proof of \cref{thm:main:beta} i)
}
We now carry out the proof of \cref{thm:main:beta}. 
Let $T^*<\infty$ be the maximal time of existence for the solution $\tht$ of \eqref{gSQG} as obtained in  \cref{thm:main:beta:recall}. We show that if 
\begin{align}\label{blowup:finite}
    \int_0^{T^*}\Sob{\Lam^\al \tht(s)}{H^{1+\be-2\al}}^2\,ds<\infty,
\end{align}
then we arrive at a contradiction. Assuming \eqref{blowup:finite} is true, then for each $\eps>0$, there exists a time $T_0\in (0,T^{*})$ satisfying
\begin{align}\label{blowup:finite:epsilon}
    \int_{T_0}^{T^*}\Sob{\Lam^\al \tht(s)}{H^{1+\be-2\al}}^2\,ds<\eps.
\end{align}
The proof is divided into three steps.
\subsection{Step 1}
First, we claim that \eqref{blowup:finite:epsilon} implies that $\Sob{\tht(t)}{H^{1+\be-2\al}}$ is bounded for all $t\in [0,T^{*})$. To prove the claim, we begin by taking the $L^2-$ inner product in \eqref{gSQG} and obtain
\begin{align}\label{L2}
    \frac{1}{2}\frac{d}{dt}\Sob{\tht(t)}{L^2}^2+\Sob{\Lam^{\al}\tht(t)}{L^2}^2=0.
\end{align}
Applying the Littlewood-Paley dyadic operator $\Delta_j$ to \eqref{gSQG} and taking the inner product with $\Delta_j (-\Delta)^{1+\be-2\al} \tht$, we obtain
\begin{align}\label{Lamj:critical}
    \frac{1}{2}\frac{d}{dt}\Sob{\Lam_j^{1+\be-2\al}\tht(t)}{L^2}^2+\Sob{\Lam^{1+\be-\al}_j\tht(t)}{L^2}^2 \le|\lbn \Delta_j(u_\tht\cdot \nabla \tht),\Delta_j (-\Delta)^{1+\be-2\al} \tht\rbn_{L^2}|,\quad \forall t\in[0,T^{*}),
\end{align}
where we employed the notation defined in \eqref{def:Lamj:notation}.
Now, let $t_1$ and $t_2$ be such that $T_0\le t_1\le t_2<T^*$. We integrate the above inequality over $[T_0, t_1]$ to obtain
\begin{align}\label{Hdot:critical}
    &\Sob{\Lam_j^{1+\be-2\al}\tht(t_1)}{L^2}^2+2\int_{T_0}^{t_1}\Sob{\Lam_j^{1+\be-\al}\tht(s)}{L^2}^2\,ds\notag \\&\le \Sob{\Lam_j^{1+\be-2\al}\tht(T_0)}{L^2}^2+2\int_{T_0}^{t_1} |\lbn \Delta_j(u_\tht\cdot \nabla \tht),\Delta_j (-\Delta)^{1+\be-2\al} \tht\rbn_{L^2}|\,ds.
\end{align}
Summing over $j\in \mathbb{Z}$ and applying \eqref{norm:equivalence} followed by an application of \cref{lem:main:particular} results in
\begin{align}\label{H:critical}
    \Sob{\tht(t_1)}{\Hdot^{1+\be-2\al}}^2+2\int_{T_0}^{t_1}\Sob{\tht}{\Hdot^{1+\be-\al}}^2\,ds\le \Sob{\tht(T_0)}{\Hdot^{1+\be-2\al}}^2+C(\al,\be)\int_{T_0}^{t_1}\Sob{\tht}{\dot{H}^{1+\be-2\al}}\Sob{\tht}{\dot{H}^{1+\be-\al}}^2\,ds.
\end{align}
We integrate \eqref{L2} over $[T_0, t_1]$, add to  the above inequality and use the fact that $1+\be-2\al\ge 0$ to obtain
\begin{align*}
    \Sob{\tht(t_1)}{H^{1+\be-2\al}}^2+2\int_{T_0}^{t_1}\Sob{\Lam^{\al}\tht}{H^{1+\be-2\al}}^2\,ds\le \Sob{\tht(T_0)}{H^{1+\be-2\al}}^2+C(\al,\be)\int_{T_0}^{t_1}\Sob{\tht}{{H}^{1+\be-2\al}}\Sob{\Lam^{\al}\tht}{{H}^{1+\be-2\al}}^2\,ds.
\end{align*}
Since $\tht\in C([T_0, t_2];H^{1+\be-2\al})$, it follows that
\begin{align*}
    &\Sob{\tht(t_1)}{H^{1+\be-2\al}}^2+2\int_{T_0}^{t_1}\Sob{\Lam^{\al}\tht}{H^{1+\be-2\al}}^2\,ds \\&\le \Sob{\tht(T_0)}{H^{1+\be-2\al}}^2+C(\al,\be)\sup_{s\in[T_0, t_2]}\left\{\Sob{\tht(s)}{H^{1+\be-2\al}}\right\}\times \int_{T_0}^{T^*}\Sob{\Lam^{\al}\tht}{H^{1+\be-\al}}^2\,ds.
\end{align*}
We take $\epsilon=C(\al,\be)^{-1}$ in \eqref{blowup:finite:epsilon} and apply the Young's inequality to obtain
\begin{align*}
    \Sob{\tht(t_1)}{H^{1+\be-2\al}}^2 \le \Sob{\tht(T_0)}{H^{1+\be-2\al}}^2+\frac{1}{2}+\frac{1}{2}\sup_{s\in[T_0, t_2]}\left\{\Sob{\tht(s)}{H^{1+\be-2\al}}^2\right\}.
\end{align*}
Taking supremum over $[T_0, t_2]$ on the left side, we obtain
\begin{align*}
    \sup_{s\in[T_0, t_2]}\left\{\Sob{\tht(s)}{H^{1+\be-2\al}}^2\right\}\le 2\Sob{\tht(T_0)}{H^{1+\be-2\al}}^2+1.
\end{align*}
Since $\tht \in C([0,T_0];H^{1+\be-2\al})$, we have 
\[\sup_{s\in[0,T_0]}\left\{\Sob{\tht(s)}{H^{1+\be-2\al}}\right\}<\infty.\]
Denote by
\[K_{T_0}=\max \left\{\left(2\Sob{\tht(T_0)}{H^{1+\be-2\al}}^2+1\right)^{\frac{1}{2}}, \sup_{s\in[0,T_0]}\left\{\Sob{\tht(s)}{H^{1+\be-2\al}}\right\}\right\}.\]
Then, we have
\begin{align}\label{Hcritical}
    \Sob{\tht(t)}{H^{1+\be-2\al}}\le K_{T_0},\quad \forall t\in [0,T^*).
\end{align}
\subsection{Step 2} Next, we claim that there exists a function $\tilde{\tht}\in H^{1+\beta-2\al}$ such that 
\begin{align*}
    \lim_{t \uparrow T^*}\Sob{\tht(t)-\tilde{\tht}}{L^2}=0.
\end{align*}
We consider a sequence $\{\lambda_n\}_{n\in \mathbb{N}}$ such that $\lambda_n \in [0,T^*)$ for all $n\in\mathbb{N}$ and
$\lim_{n\to \infty}\lambda_n=T^*.$
We show that $\{\lambda_n\}$ is Cauchy is $L^2$. Integrate \eqref{gSQG} over the intervals $[0,\lambda_n]$ and $[0,\lambda_m]$ and take the difference of the resulting equations to obtain
\begin{align}
    \tht(\lambda_n)-\tht(\lambda_m)=-\int_{\lambda_m}^{\lambda_n}((-\Delta)^{\al}\tht)(s)\,ds-\int_{\lambda_m}^{\lambda_n}(u_\tht \cdot \nabla \tht)(s)\,ds.
\end{align}
Taking $L^2-$ norm on both sides and using Plancherel's identity, we have
\begin{align}
    \Sob{\tht(\lambda_n)-\tht(\lambda_m)}{L^2}\le \int_{\lambda_m}^{\lambda_n}\Sob{((-\Delta)^{\al}\tht)(s)}{L^2}\,ds+\int_{\lambda_m}^{\lambda_n}\Sob{(u_\tht \cdot \nabla \tht)(s)}{L^2}\,ds.
\end{align}
Since $\nabla^{\perp}\tht\cdot \nabla \tht=0$, we have
\begin{align}\label{commutator:decomposition}
    u_\tht \cdot \nabla \tht=\nabla^{\perp}\Lam^{\be-2}\tht\cdot \nabla \tht+\Lam^{\be-2}(\nabla^{\perp}\tht\cdot \nabla \tht)=[\nabla^{\perp}\Lam^{\be-2}\cdot,\nabla \tht]\tht.
\end{align}
Applying \cref{lem:commutator1} with $s=2-\be$, $\rho_1=\rho_2=2\al$, and using the fact that $2\al\le 1+\be-2\al$, we obtain
\begin{align}
    \Sob{\tht(\lambda_n)-\tht(\lambda_m)}{L^2}&\le \int_{\lambda_m}^{\lambda_n}\Sob{\tht(s)}{\Hdot^{2\al}}\,ds+C(\al,\be)\int_{\lambda_m}^{\lambda_n}\Sob{\nabla \tht}{\Hdot^{\be-2\al}}\Sob{\tht}{\Hdot^{2\al}}\,ds\notag\\&\le \int_{\lambda_m}^{\lambda_n}\Sob{\tht(s)}{\Hdot^{2\al}}\,ds+C(\al,\be)\int_{\lambda_m}^{\lambda_n}\Sob{\tht}{\Hdot^{1+\be-2\al}}\Sob{\tht}{\Hdot^{2\al}}\,ds \notag\\&\le
    \left(\int_{\lambda_m}^{\lambda_n}\Sob{\tht(s)}{H^{1+\be-2\al}}\,ds+C(\al,\be)\int_{\lambda_m}^{\lambda_n}\Sob{\tht(s)}{H^{1+\be-2\al}}^2\,ds\right)\notag\\
    &\le \left(K_{T_0}+C(\al,\be)K_{T_0}^2\right)|\lam_n-\lam_m|.
\end{align}
Since $\lam_n \to T^*$, we obtain that $\{\tht(\lam_n)\}$ is Cauchy in $L^2$. As a result, there exists $\tilde{\tht} \in L^2$ such that 
\begin{align}\label{lim:L2}
    \lim_{n\to \infty}\Sob{\tht(\lam_n)-\tilde{\tht}}{L^2}=0.
\end{align}
Additionally, from \eqref{Hcritical}, we deduce that $\{\tht(\lam_n)\}$ is a bounded sequence in $H^{1+\be-2\al}$. Therefore, we can extract a subsequence of $\{\tht(\lam_n)\}$, denoted $\{\tht(\lam_{n_l})\}$ such that
\begin{align}\label{lam:subsequence}
    \tht(\lam_{n_l})\rightharpoonup \bar{\tht}\quad \text{in}\quad H^{1+\be-2\al}.
\end{align}
From this, we deduce that $\tht(\lam_n)\rightharpoonup \bar{\tht}$ in $L^2$ and hence $\tilde{\tht}=\bar{\tht}$.
It remains to prove that
\[\lim_{t \uparrow T^*}\Sob{\tht(t)-\tilde{\tht}}{H^{1+\be-2\al}}=0.\] We argue by contradiction. Assume that the above limit does not hold. Then there exists $\epsilon>0$ and an sequence $\{t_n\}$ such that $t_n\in ((1-T^*/n,T^*)$ for all $n\in \mathbb{N}$ and 
\begin{align}\label{tn:Hcritical}
    \Sob{\tht(t_n)-\tilde{\tht}}{H^{1+\be-2\al}}\ge \epsilon,\quad \forall n\in \mathbb{N}.
\end{align}
Similar to \eqref{lam:subsequence}, we can extract a subsequence of $\{\tht(t_n)\}$, denoted $\{\tht(t_{n_l})\}$ such that
\begin{align}\label{weak:limit:theta:tnl}
    \tht(t_{n_l})\rightharpoonup \tilde{\tht}\quad \text{in}\quad H^{1+\be-2\al}.
\end{align}
Using interpolation, Plancherel's identity and \eqref{Hcritical}, we have for all $r\in [0,1+\be-2\al)$
\begin{align}\label{Hr:interpolation}
    \Sob{\tht(t)-\tilde{\tht}}{H^r}&\le C(\al,\be,\mu)\Sob{\tht(t)-\tilde{\tht}}{L^2}^{\mu}\Sob{\tht(t)-\tilde{\tht}}{H^{1+\be-2\al}}^{1-\mu}\notag\\
    &\le C(\al,\be,\mu)\Sob{\tht(t)-\tilde{\tht}}{L^2}^{\mu}\left(\Sob{\tht(t)}{H^{1+\be-2\al}}+\Sob{\tilde{\tht}}{H^{1+\be-2\al}}\right)^{1-\mu}\notag\\
    &\le C(\al,\be,\mu)\Sob{\tht(t)-\tilde{\tht}}{L^2}^{\mu}\left(K_{T_0}+\Sob{\tilde{\tht}}{H^{1+\be-2\al}}\right)^{1-\mu},
\end{align}
where $\mu\in (0,1]$. From \eqref{lim:L2} and \eqref{Hr:interpolation}, we obtain 
\begin{align*}
     \lim_{t \uparrow T^*}\Sob{\tht(t)-\tilde{\tht}}{H^r}=0, \quad \forall r\in [0,1+\be-2\al).
\end{align*}
\subsection{Step 3} Now we prove the claim stated in \eqref{regularity:criterion}. Let $\{\gamma_n\}_{n\in\mathbb{N}}$ be a sequence in $ (2-\al,1+\be-2\al)$ such that $\lim_{n\to \infty}\gamma_n=1+\be-2\al$.
Applying the Littlewood-Paley dyadic operator $\Delta_j$ to \eqref{gSQG} and taking the inner product with $\Delta_j (-\Delta)^{\gamma_n} \tht$, we obtain
\begin{align*}
    \frac{1}{2}\frac{d}{dt}\Sob{\Lam_j^{\gamma_n}\tht(t)}{L^2}^2+\Sob{\Lam^{\gamma_n+\al}_j\tht(t)}{L^2}^2 \le|\lbn \Delta_j(u_\tht\cdot \nabla \tht),\Delta_j (-\Delta)^{\gamma_n} \tht\rbn_{L^2}|.
\end{align*}
Now let $t_3$ and $t_4$ be such that $0\le t_3\le t_4<T^*$. We integrate the above inequality over $[t_3, t_4]$ to obtain
\begin{align*}
    \Sob{\Lam_j^{\gamma_n}\tht(t_4)}{L^2}^2+2\int_{t_3}^{t_4}\Sob{\Lam_j^{\gamma_n+\al}\tht(s)}{L^2}^2\,ds\le \Sob{\Lam_j^{\gamma_n}\tht(t_3)}{L^2}^2+2\int_{t_3}^{t_4} |\lbn \Delta_j(u_\tht\cdot \nabla \tht),\Delta_j (-\Delta)^{\gamma_n} \tht\rbn_{L^2}|\,ds.
\end{align*}
Summing over $j\in \mathbb{Z}$ and applying \eqref{norm:equivalence} followed by an application of \cref{lem:main} with $\si=\gamma_n$ gives us
\begin{align*}
    \Sob{\tht(t_4)}{\Hdot^{\gamma_n}}^2+2\int_{t_3}^{t_4}\Sob{\tht}{\Hdot^{\gamma_n+\al}}^2\,ds\le \Sob{\tht(t_3)}{\Hdot^{\gamma_n}}^2+C(\al,\be)\int_{t_3}^{t_4}\Sob{\tht}{\dot{H}^{\gamma_n}}\Sob{\tht}{\dot{H}^{\gamma_n+\al}}^2\,ds.
\end{align*}
We integrate \eqref{L2} over $[t_3, t_4]$ and then add to the above inequality, and use the fact that $\gamma_n < 1+\be-2\al$ to obtain
\begin{align*}
    \Sob{\tht(t_3)}{H^{\gamma_n}}^2+2\int_{t_3}^{t_4}\Sob{\Lam^{\al}\tht}{H^{\gamma_n}}^2\,ds\le \Sob{\tht(t_4)}{H^{\gamma_n}}^2+C(\al,\be)\int_{t_3}^{t_4}\Sob{\tht}{{H}^{1+\be-2\al}}\Sob{\Lam^{\al}\tht}{{H}^{1+\be-2\al}}^2\,ds.
\end{align*}
Taking the limit as $t_4 \to T^*$ in the above inequaliy, we obtain
\begin{align}\label{H:gamman}
    \Sob{\tht(t_3)}{H^{\gamma_n}}^2\le \Sob{\tilde{\tht}}{H^{\gamma_n}}^2+C(\al,\be)K_{T_0}\int_{t_3}^{T^*}\Sob{\Lam^{\al}\tht}{{H}^{1+\be-2\al}}^2\,ds
\end{align}
Since $\lim_{n\to\infty}\gamma_n=1+\be-2\al$, we apply the dominated convergence theorem and deduce that for each $t\in [0,T^*)$, we have
\begin{align*}
  \lim_{n\to \infty}\Sob{\tht(t)}{H^{\gamma_n}}=\Sob{\tht(t)}{H^{1+\be-2\al}}. 
\end{align*}
Similarly, we also have
\[\lim_{n\to \infty}\Sob{\tilde{\tht}}{H^{\gamma_n}}=\Sob{\tilde{\tht}}{H^{1+\be-2\al}}.\]
Now we pass to the limit, as $n\to \infty$, in \eqref{H:gamman} to obtain
\begin{align*}
    \Sob{\tht(t)}{H^{1+\be-2\al}}^2\le \Sob{\tilde{\tht}}{H^{1+\be-2\al}}^2+C(\al,\be)K_{T_0}\int_{t}^{T^*}\Sob{\Lam^{\al}\tht}{{H}^{1+\be-2\al}}^2\,ds,\quad \forall t \in [0,t^*).
\end{align*}
In particular, we have
\begin{align*}
    \Sob{\tht(t_{n_l})}{H^{1+\be-2\al}}^2\le \Sob{\tilde{\tht}}{H^{1+\be-2\al}}^2+C(\al,\be)K_{T_0}\int_{t_{n_l}}^{T^*}\Sob{\Lam^{\al}\tht}{{H}^{1+\be-2\al}}^2\,ds,\quad \forall t \in [0,t^*),\quad \forall l \in \mathbb{N}.
\end{align*}
Taking $\limsup_{l\to \infty}$ in the above inequality, we obtain
\begin{align}\label{lam:subsequence}
    \limsup_{l\to \infty}\Sob{\tht(t_{n_l})}{H^{1+\be-2\al}}\le \Sob{\tilde{\tht}}{H^{1+\be-2\al}}.
\end{align}
By \eqref{weak:limit:theta:tnl}, \eqref{lam:subsequence}, and \cref{lem:Hilbert:conv}, we deduce that
\begin{align*}
    \lim_{l\to \infty}\Sob{\tht(t_{n_l})-\tilde{\tht}}{H^{1+\be-2\al}}=0,
\end{align*}
which contradicts \eqref{tn:Hcritical}. Therefore, we obtain
\begin{align*}
    \lim_{t \uparrow T^*}\Sob{\tht(t)-\tilde{\tht}}{H^{1+\be-2\al}}=0.
\end{align*}
Now we initialize \eqref{gSQG} with data $\tilde{\tht}$. By \cref{thm:main:beta:recall}, there exists a time $T_1\in (0,\infty]$ and a unique solution $\tht_1(\cdot;\tilde{\tht})\in C([0,T_1];H^{1+\be-2\al})$. 
Now, if we define the function
\begin{align*}
    \tht_2=\begin{cases}
        \tht(t),\quad &t\in[0,T^*);\\
        \tht_1(t-T^*) &t\in[T^*, T_1+T^*).
    \end{cases}
\end{align*}
Then, it is straightforward to verify that $\tht_2$ is a solution of \eqref{gSQG} lying in $C([0,T_1+T^*);H^{1+\be-2\al})$ which contradicts the assumption that $T^*$ is the maximal time of existence for \eqref{gSQG}. As a result, we conclude that the integral in \eqref{blowup:finite} is not finite.
\section{Proof of \cref{thm:main:beta} ii)}
Let $\tht \in C([0,T^*);H^{1+\be-2\al})$ be the maximal solution obtained in \cref{thm:main:beta:recall}. To establish the existence of a unique solution satisfying the energy inequality \eqref{energy:inequality}, we assume that the initial data $\tht_0$ satisfies
\[\Sob{\tht_0}{H^{1+\be-2\al}}<(4C(\al,\be))^{-1},\]
where $C(\al,\be)$ is the constant in \cref{lem:main:particular}.
We proceed just like in \eqref{Lamj:critical}-\eqref{H:critical} to obtain
\begin{align}\label{H:critical:proof2}
    \Sob{\tht(t)}{H^{1+\be-2\al}}^2+2\int_{0}^{t}\Sob{\Lam^{\al}\tht}{H^{1+\be-2\al}}^2\,ds\le \Sob{\tht_0}{H^{1+\be-2\al}}^2+2C(\al,\be)\int_{0}^{t}\Sob{\tht}{{H}^{1+\be-2\al}}\Sob{\Lam^{\al}\tht}{{H}^{1+\be-2\al}}^2\,ds
\end{align}
Let us define
\[T_2:=\sup \left\{t\in [0,T^*):\,\sup_{s\in[0,t]} \Sob{\tht(s)}{H^{1+\be-2\al}}<2\Sob{\tht_0}{H^{1+\be-2\al}}\right\}.\]
Since $\tht \in C([0,T^*);H^{1+\be-2\al})$, we have
\[T_2 \in (0,T^*)\]
Now, from \eqref{H:critical} and the definition of $T_2$, we have 
\begin{align*}
    \Sob{\tht(t)}{H^{1+\be-2\al}}^2+2\int_{0}^{t}\Sob{\Lam^{\al}\tht}{H^{1+\be-2\al}}^2\,ds&\le \Sob{\tht_0}{H^{1+\be-2\al}}^2+4C(\al,\be)\Sob{\tht_0}{H^{1+\be-2\al}}\int_0^t\Sob{\Lam^{\al}\tht}{{H}^{1+\be-2\al}}^2\,ds\\
    &\le \Sob{\tht_0}{H^{1+\be-2\al}}^2+\int_0^t\Sob{\Lam^{\al}\tht}{{H}^{1+\be-2\al}}^2\,ds.
\end{align*}
for all $t\in[0,T_2)$.
This implies that
\begin{align}\label{H:critical:required}
    \Sob{\tht(t)}{H^{1+\be-2\al}}^2+\int_{0}^{t}\Sob{\Lam^{\al}\tht}{H^{1+\be-2\al}}^2\,ds\le \Sob{\tht_0}{H^{1+\be-2\al}}^2
\end{align}
for all $t\in [0,T_2)$. To prove the claim that $T_2=T^*$, we argue by contradiction. Assume that $T_2<T^*$. Since $\tht \in C([0,T^*);H^{1+\be-2\al})$, we can pass to the limit in \eqref{H:critical:required} and obtain 
\begin{align*}
    \Sob{\tht(T_2)}{H^{1+\be-2\al}} \le \Sob{\tht_0}{H^{1+\be-2\al}}.
\end{align*}
We again use the fact that $\tht \in C([0,T^*);H^{1+\be-2\al})$ to deduce the existence of a time $T_3\in (T_2,T^*)$ such that 
\begin{align}\label{H:critical:T2-T3}
    \Sob{\tht(s)}{H^{1+\be-2\al}}<2\Sob{\tht_0}{H^{1+\be-2\al}},\quad \forall s\in[T_2,T_3].
\end{align}
From \eqref{H:critical:T2-T3} and the definition of $T_2$, we conclude
\begin{align*}
     \Sob{\tht(s)}{H^{1+\be-2\al}}<2\Sob{\tht_0}{H^{1+\be-2\al}},\quad \forall s\in[0,T_3].
\end{align*}
As a result, we have
\begin{align*}
    \sup_{s\in[0,T_3]}\Sob{\tht(s)}{H^{1+\be-2\al}}<2\Sob{\tht_0}{H^{1+\be-2\al}}.
\end{align*}
Thus $T_2\ge T_3$ which contradicts the choice of $T_3$. Therefore, we deduce that $T_2=T^*$. As a result, from \eqref{H:critical:required}, we conclude that
\begin{align*}
    \int_0^{T^*}\Sob{\Lam^{\al}\tht(s)}{H^{1+\be-2\al}}^2\,ds\le \Sob{\tht_0}{H^{1+\be-2\al}}.
\end{align*}
Therefore, by part i), we must have $T^*=\infty$. Again, by invoking \eqref{H:critical:required}, it follows that
    \begin{align}\label{H:critical:infinite:time}
    \Sob{\tht(t)}{H^{1+\be-2\al}}^2+\int_{0}^{t}\Sob{\Lam^{\al}\tht}{H^{1+\be-2\al}}^2\,ds\le \Sob{\tht_0}{H^{1+\be-2\al}}^2,\quad \forall t \in [0,\infty)
\end{align}
thus completing the proof of Theorem 1.2 ii).
\section{Proof of \cref{thm:main:beta} iii)}
Let $\tht \in C([0,\infty);H^{1+\be-2\al})\cap L^2([0,\infty);\Hdot^{1+\be-\al})$ be the solution of \eqref{gSQG} emanating from initial data $\tht_0$ which satisfies $\Sob{\tht_0}{H^{1+\be-2\al}}<(4C(\al,\be))^{-1}$, where $C(\al,\be)$ is the constant in \cref{lem:main}. The proof is divided into two steps.
\subsection{Step 1} First we prove that
\begin{align*}
    \lim_{t\to \infty}\Sob{\tht(t)}{L^2}=0.
\end{align*}
Integrating \eqref{L2} over $[r,t]$ for $t\ge r\ge 0$, we obtain
\begin{align*}
    \Sob{\tht(t)}{L^2}^2+2\int_r^t\Sob{\Lam^\al \tht(s)}{L^2}^2\,ds=\Sob{\tht(r)}{L^2}^2,
\end{align*}
For any $\delta>0$, we define
\begin{align*}
    w_{\delta}(x,t):=\mathcal{F}^{-1}(\chi_{\{|\xi|\le\delta\}}\widehat{\tht})(x,t),
\end{align*}
and
\begin{align*}
    v_\delta(x,t):=\mathcal{F}^{-1}(\chi_{\{|\xi|>\delta\}}\widehat{\tht})(x,t),
\end{align*}
for all $x\in \mathbb{R}^2$ and $t\ge 0$. Then we can see that $w_\delta(x,t)$ evolves according to
\begin{align*}
    \bdy_t w_\delta+(-\Delta)^{\al}w_\delta+ \mathcal{F}^{-1}(\chi_{\{|\xi|\le\delta\}}\mathcal{F}(u_\tht \cdot \nabla \tht))=0.
\end{align*}
We take the $L^2-$ inner product with $w_\delta$ in the above equation and use Plancherel's identity and the Cauchy-Schwarz inequality to obtain
\begin{align*}
    \frac{1}{2}\frac{d}{dt}\Sob{w_\delta(t)}{L^2}^2+\Sob{\Lam^{\al}w_\delta(t)}{L^2}^2&\le |\lbn w_\delta,\mathcal{F}^{-1}(\chi_{\{|\xi|\le \delta\}}\mathcal{F}(u_\tht \cdot \nabla \tht)) \rbn_{L^2}|\\
    &=|\lbn \widehat{w_\delta},\chi_{\{|\xi|\le \delta\}}\mathcal{F}(u_\tht \cdot \nabla \tht) \rbn|\\
    &\le \int_{|\xi|\le \delta}|\widehat{w_\delta}||\mathcal{F}(u_\tht \cdot \nabla \tht)|\,d\xi\\
    &\le \delta^{\kappa}\int_{|\xi|\le \delta}|\widehat{w_\delta}||\xi|^{-\kappa}|\mathcal{F}(u_\tht \cdot \nabla \tht)|\,d\xi, 
\end{align*}
for $\kap>0$. Applying Holder's inequality, Plancherel's identity, and using \eqref{commutator:decomposition}, we have
\begin{align}\label{w:delta}
    \frac{1}{2}\frac{d}{dt}\Sob{w_\delta(t)}{L^2}^2+\Sob{\Lam^{\al}w_\delta(t)}{L^2}^2&\le \delta^{\kappa}\Sob{w_\delta}{L^2}\Sob{[\nabla^{\perp}\Lam^{\be-2}\cdot,\nabla \tht]\tht}{\Hdot^{-\kappa}}.
\end{align}
Let $\kappa \in(0,1)$. Applying \cref{lem:commutator1} with $s=2-\be$, $\rho_1=\al+\kappa,\,\rho_2=\al$, we have
\begin{align}\label{w:delta:2}
     \frac{1}{2}\frac{d}{dt}\Sob{w_\delta(t)}{L^2}^2+\Sob{\Lam^{\al}w_\delta(t)}{L^2}^2&\le \delta^{\kappa}C(\al,\be,\kappa)\Sob{w_\delta}{L^2}\Sob{ \tht}{\Hdot^{1+\be-\al-\kappa}}\Sob{\tht}{\Hdot^{\al}}\notag\\
     &\le \delta^{\kappa}C(\al,\be,\kappa)\Sob{w_\delta}{L^2}\Sob{ \Lam^{\al}\tht}{\Hdot^{1+\be-2\al-\kappa}}\Sob{\Lam^\al\tht}{L^2}\notag\\
     &\le \delta^{\kappa}C(\al,\be,\kappa)\Sob{w_\delta}{L^2}\Sob{\Lam^{\al}\tht}{H^{1+\be-2\al}}\Sob{\Lam^\al\tht}{L^2}, 
\end{align}
where the last step follows since $1+\be-2\al\ge 2$. By Plancherel's identity, we have
\begin{align}\label{w:delta:L2}
    \Sob{w_\delta}{L^2}^2=\int_{\{|\xi|\le \delta\}}|\widehat{\tht}|^2\,d\xi\le \Sob{\widehat{\tht}}{L^2}^2=\Sob{\tht}{L^2}^2\le \Sob{\tht_0}{L^2}^2.
\end{align}
We integrate \eqref{w:delta:2} over $[0,t]$, use \eqref{w:delta:L2}, Plancherel's identity, and the Cauchy-Schwarz inequality to obtain
\begin{align*}
&\Sob{w_\delta}{L^2}^2+2\int_0^t\Sob{\Lam^{\al}w_\delta (s)}{L^2}^2\,ds\\&\le \Sob{w_\delta(0)}{L^2}^2+\delta^{\kappa}C(\al,\be,\kappa)\Sob{\tht_0}{L^2}\int_0^t \Sob{\Lam^{\al}\tht}{H^{1+\be-2\al}}\Sob{\Lam^\al\tht}{L^2}\\
&\le \Sob{w_\delta(0)}{L^2}^2+\delta^{\kappa}C(\al,\be,\kappa)\Sob{\tht_0}{L^2}\left(\int_0^t \Sob{\Lam^{\al}\tht}{H^{1+\be-2\al}}^2\,ds\right)^{\frac{1}{2}}\left(\int_0^t\Sob{\Lam^\al\tht}{L^2}^2\,ds\right)^{\frac{1}{2}}.
\end{align*}
From \eqref{L2} and \eqref{H:critical:infinite:time}, we have
\begin{align}\label{theta:L2:H}
\int_0^t \Sob{\Lam^{\al}\tht}{H^{1+\be-2\al}}^2\,ds\le \Sob{\tht_0}{H^{1+\be-2\al}}^2,\quad\text{and}\quad\int_0^t\Sob{\Lam^\al\tht}{L^2}^2\,ds\le \Sob{\tht_0}{L^2}^2\end{align}
for all $t\ge0$. Thus, we obtain
\begin{align*}
    \Sob{w_\delta(t)}{L^2}\le \Gamma_{\delta}, \quad \forall t\ge 0,
\end{align*}
where 
\begin{align*}
    \Gamma_{\delta}:=\left(\Sob{w_\delta(0)}{L^2}^2+\delta^{\kappa}C(\al,\be,\kappa)\Sob{\tht_0}{L^2}^2\Sob{\tht_0}{H^{1+\be-2\al}}\right)^{\frac{1}{2}}.
\end{align*}
Also,  
\begin{align*}
    \int_0^\infty\Sob{\Lam^\al w_\delta(t)}{L^2}^2\,ds\le \frac{1}{2}\Gamma_{\delta}.
\end{align*}
Clearly, we have
\begin{align*}
    \lim_{\delta\to 0}\Gamma_{\delta}=0,
\end{align*}
since $\kappa>0$. As a result, for each $\eps>0$, there exists $\delta_1>0$ such that 
\begin{align}\label{w:delta1:small}
    \Sob{w_{\delta_1}(t)}{L^2}<\frac{\eps}{2\sqrt{2}}, \quad \int_0^\infty\Sob{\Lam^{\al}w_{\delta_1}(s)}{L^2}^2<\eps,
\end{align}
for all $t\ge 0$. Now we study the evolution of  $v_{\delta_1}(x,t)$, which satisfies
\begin{align*}
    \bdy_t v_{\delta_1}+(-\Delta)^{\al}v_{\delta_1}+ \mathcal{F}^{-1}(\chi_{\{|\xi|>\delta_1\}}\mathcal{F}(u_\tht \cdot \nabla \tht))=0.
\end{align*}
Applying $e^{s(-\Delta)^{\al}}$ (with $s\in[0,t]$) to the above equation, and integrating over $[0,t]$, we obtain
\begin{align*}
    v_{\delta_1}(t)=e^{-t(-\Delta)^{\al}}v_{\delta_1}(0)-\int_0^t e^{-(t-s)(-\Delta)^{\al}}\mathcal{F}^{-1}\left(\chi_{\{|\xi|>\delta_1\}}\mathcal{F}(u_\tht \cdot \nabla \tht)(s)\right)\,ds.
\end{align*}
Let $\kappa \in (0,1)$. We take $\Hdot^{-\kappa}-$ norm in the above equality to obtain
\begin{align}\label{v:delta1}
    \Sob{v_{\delta_1}(t)}{\Hdot^{-\kappa}}\le &\Sob{e^{-t(-\Delta)^{\al}}v_{\delta_1}(0)}{\Hdot^{-\kappa}}\notag\\
    &+ \int_0^t \Sob{e^{-(t-s)(-\Delta)^{\al}}\mathcal{F}^{-1}\left(\chi_{\{|\xi|>\delta_1\}}\mathcal{F}(u_\tht \cdot \nabla \tht))\right)}{\Hdot^{-\kappa}}\,ds.
\end{align}
 The first term on the right in the above inequality is estimated as 
 \begin{align*}
     \Sob{e^{-t(-\Delta)^{\al}}v_{\delta_1}(0)}{\Hdot^{-\kappa}}^2&=\int_{\mathbb{R}^2}|\xi|^{-2\kappa}e^{-2t|\xi|^{2\al}}|\widehat{v_{\delta_1}}(0)|^2\,d\xi\\
     &=\int_{|\xi|>\delta_1}|\xi|^{-2\kappa}e^{-2t|\xi|^{2\al}}|\widehat{\tht_{0}}|^2\,d\xi\\
     &\le \delta_1^{-2\kappa}e^{-2t\delta_1^{2\al}}\Sob{\widehat{\tht_{0}}}{L^2}^2
 \end{align*}
Using Plancherel's identity, we conclude
\begin{align*}
    \Sob{e^{-t(-\Delta)^{\al}}v_{\delta_1}(0)}{\Hdot^{-\kappa}}\le \delta_1^{-\kappa}e^{-t\delta_1^{2\al}}\Sob{\tht_0}{L^2}. 
\end{align*}
The second term on the right of \eqref{v:delta1} is estimated by using \eqref{commutator:decomposition} and applying \cref{lem:commutator1} with $s=2-\be$, $\rho_1=\al+\kappa$ and $\rho_2=\al$ to obtain
\begin{align*}
    &\int_0^t \Sob{e^{-(t-s)(-\Delta)^{\al}}\mathcal{F}^{-1}\left(\chi_{\{|\xi|>\delta_1\}}\mathcal{F}(u_\tht \cdot \nabla \tht))\right)}{\Hdot^{-\kappa}}\,ds\\
    &\le \int_0^t e^{-(t-s)\delta_1^{2\al}}\left(\int_{\{|\xi|>\delta_1\}}|\xi|^{-2\kappa}|\mathcal{F}(u_{\tht} \cdot \nabla \tht)|^2\right)^{\frac{1}{2}}\,ds\\
    &\le C(\al,\be,\kappa)\int_0^{t}e^{-(t-s)\delta_1^{2\al}}\Sob{\tht}{\Hdot^{1+\be-\al-\kappa}}\Sob{\tht}{\Hdot^{\al}}\,ds\\
    &\le C(\al,\be,\kappa)\int_0^{t}e^{-(t-s)\delta_1^{2\al}}\Sob{\Lam^{\al}\tht}{H^{1+\be-2\al}}\Sob{\Lam^{\al}\tht}{L^2}\,ds.
\end{align*}
Taking the $L^2([0,\infty))-$ norm in \eqref{v:delta1}, we obtain
\begin{align*}
    &\Sob{v_{\delta_1}}{L^2([0,\infty);\Hdot^{-\kappa})}\\&\le \delta_1^{-\kappa}\Sob{\tht_0}{L^2}\left(\int_0^\infty e^{-2t\delta_{1}^{2\al}}\,dt\right)^{\frac{1}{2}}\\&+C(\al,\be,\kappa)\left(\int_0^\infty\left(\int_0^\infty\chi_{(s,\infty)}(t)e^{-(t-s)\delta_1^{2\al}}\Sob{\Lam^\al \tht}{H^{1+\be-2\al}}\Sob{\Lam^\al \tht}{L^2}\,ds\right)^2\,dt\right)^{\frac{1}{2}}\\
    &\le \frac{1}{\sqrt{2}}\delta_1^{-\kappa-\al}\Sob{\tht_0}{L^2}+C(\al,\be,\kappa)\int_0^\infty e^{s\delta_1^{2\al}}\Sob{\Lam^{\al}\tht(s)}{H^{1+\be-2\al}}\Sob{\Lam^{\al}\tht(s)}{L^2}\left(\int_s^\infty e^{-2t\delta_1^{2\al}}\,dt\right)^{\frac{1}{2}}\,ds\\
    &\le \frac{1}{\sqrt{2}}\delta_1^{-\kappa-\al}\Sob{\tht_0}{L^2}+C(\al,\be,\kappa)\delta_1^{-\al}\int_0^\infty \Sob{\Lam^{\al}\tht(s)}{H^{1+\be-2\al}}\Sob{\Lam^{\al}\tht(s)}{L^2}\,ds.
\end{align*}
Applying the Cauchy-Schwarz inequality and using \eqref{L2} and \eqref{H:critical:infinite:time}, we obtain
\begin{align*}
    \Sob{v_{\delta_1}}{L^2([0,\infty);\Hdot^{-\kappa})}&\le \frac{1}{\sqrt{2}}\delta_1^{-\kappa-\al}\Sob{\tht_0}{L^2}+C(\al,\be,\kappa)\delta_1^{-\al}\left(\int_0^\infty \Sob{\Lam^{\al}\tht(s)}{H^{1+\be-2\al}}^2\,ds\right)^{\frac{1}{2}}\left(\int_0^\infty \Sob{\Lam^{\al}\tht(s)}{L^2}^2\,ds\right)^{\frac{1}{2}}\\
    &\le \frac{1}{\sqrt{2}}\delta_1^{-\kappa-\al}\Sob{\tht_0}{L^2}+C(\al,\be,\kappa)\delta_1^{-\al}\Sob{\tht_0}{H^{1+\be-2\al}}\Sob{\tht_0}{L^2}.
\end{align*}
Therefore, we conclude that $v_{\delta_1}\in L^{2}([0,\infty);\Hdot^{-\kappa})$. Since $\tht=w_{\delta_1}+v_{\delta_1}$, we  also have that $v_{\delta_1}=\tht-w_{\delta_1}\in L^{2}{([0,\infty);\Hdot^{\al})}$ by \eqref{theta:L2:H} and \eqref{w:delta1:small}. By interpolation inequality, we have
\begin{align*}
    \Sob{v_{\delta_1}}{L^{2}([0,\infty);L^2)}\le \Sob{v_{\delta_1}}{L^{2}([0,\infty);\Hdot^{-\kappa})}^{\frac{\al}{\kappa+\al}}\Sob{v_{\delta_1}}{L^{2}([0,\infty);\Hdot^\al)}^{\frac{\kappa}{\kappa+\al}}.
\end{align*}
As a result, we have $v_{\delta_1}\in L^{2}([0,\infty);L^2)$. Using this inclusion, we now define the following set:
\begin{align*}
    A(\eps)=\left\{t\ge 0:\Sob{v_{\delta_1}(t)}{L^2}\ge \frac{\eps}{2\sqrt{2}}\right\}\subset [0,\infty).
\end{align*}
We observe that
\begin{align}\label{T:epslion:def}
    |A(\eps)|=&\int_{A(\eps)}\,dt\le \frac{8}{\eps^2}\int_{A(\eps)}\Sob{v_{\delta_1(t)}}{L^2}^2\,dt
    \le \frac{8}{\eps^2}\int_0^\infty \Sob{v_{\delta_1(t)}}{L^2}^2=:T_\eps,
\end{align}
where $|\cdot|$ denotes the Lebesgue's measure. Since $v_{\delta_1}\in L^2([0,\infty);L^2)$, we have that $T_\eps<\infty$. From \eqref{T:epslion:def}, we can deduce that $([0,T_\eps +1]\backslash A(\eps))\ne \phi;$. Therefore, there exists $t_\eps \in ([0,T_\eps +1]\backslash A(\eps))$ which satisfies $\Sob{v_{\delta_1}(t_\eps)}{L^2}<\frac{\eps}{2\sqrt{2}}$. As a result, by \eqref{w:delta1:small}, we have 
\begin{align}\label{theta:T:epsilon}
    \Sob{\tht(t_\eps)}{L^2}\le\Sob{w_{\delta_1}(t_\eps)}{L^2}+\Sob{v_{\delta_1}(t_\eps)}{L^2}<\frac{\eps}{\sqrt{2}}.
\end{align}
Observe that
\[\Sob{\tht(t_\eps)}{H^{1+\be-2\al}}\le \Sob{\tht_0}{H^{1+\be-2\al}}<(4C(\al,\be))^{-1}.\] Now we initialize \eqref{gSQG} with data ${\tht(t_\eps)}$. Then, by part ii), there exists the unique global solution of \eqref{gSQG} satisfying
\begin{align}\label{energy:inequality:theta:T:epsilon}
    \Sob{\tht(t+t_\eps)}{L^2}^2+2\int_{t_{\eps}}^{t+t_\eps} \Sob{\Lam^{\al}\tht(s)}{L^2}^2\,ds=\Sob{\tht(t_\eps)}{L^2}^2, \quad \forall t\ge 0.
\end{align}
In particular, we have
\[\Sob{\tht(t+t_\eps)}{L^2}\le \Sob{\tht(t_\eps)}{L^2}<\eps,\quad \forall t\ge 0.\]
This implies 
\[\Sob{\tht(s)}{L^2}<\eps,\quad \forall s\ge t_\eps.\]
In other words, 
\[\lim_{t\to \infty}\Sob{\tht(t)}{L^2}=0.\]
\subsection{Step 2} Now we prove that $\lim_{t\to \infty}\Sob{\tht(t)}{H^{1+\be-2\al}}=0.$
Let\[\zeta=\frac{\al}{1+\be-\al}\]
Using interpolation, we have
\[\Sob{\tht}{\Hdot^{1+\be-2\al}}\le C\Sob{\tht}{L^2}^{\zeta}\Sob{\tht}{\Hdot^{1+\be-\al}}^{1-\zeta}.\]
As a result, we have
\begin{align*}
    \Sob{\tht}{L^{\frac{2}{1-\zeta}}([0,\infty);\Hdot^{1+\be-2\al})}^{\frac{2}{1-\zeta}}\le \int_0^\infty \Sob{\tht}{L^2}^{\frac{2\zeta}{1-\zeta}}\Sob{\tht}{\Hdot^{1+\be-\al}}^2\,dt.
\end{align*}
Using \eqref{H:critical:infinite:time}, we obtain
\begin{align*}
     \Sob{\tht}{L^{\frac{2}{1-\zeta}}([0,\infty);\Hdot^{1+\be-2\al})}\le \Sob{\tht_0}{L^2}\Sob{\tht}{L^{2}([0,\infty);\Hdot^{1+\be-\al})}^{1-
     \zeta}. 
\end{align*}
Since $\tht_0\in L^2$ and $\tht\in L^{2}([0,\infty);\Hdot^{1+\be-\al})$, we deduce that $\tht\in L^{\frac{2}{1-\zeta}}([0,\infty);\Hdot^{1+\be-2\al})$. Using this inclusion, we can define the following set:
\begin{align*}
    B(\eps)=\left\{t\ge t_\eps:\Sob{\tht(t)}{\Hdot^{1+\be-2\al}}\ge \frac{\eps}{\sqrt{2}}\right\}\subset [0,\infty),
\end{align*}
where $t_\eps$ was defined in Step 2. We observe
\begin{align}\label{T:epsilon:bar:def}
    |B(\eps)|=&\int_{B(\eps)}\,dt\le \left(\frac{\eps}{\sqrt{2}}\right)^{-\frac{2}{1-\zeta}}\int_{B(\eps)}\Sob{\tht(t)}{\Hdot^{1+\be-2\al}}^{\frac{2}{1-\zeta}}\,dt
      \le \left(\frac{\eps}{\sqrt{2}}\right)^{-\frac{2}{1-\zeta}}\int_0^\infty \Sob{\tht(t)}{\Hdot^{1+\be-2\al}}^{\frac{2}{1-\zeta}}\,dt=:\bar{T}_{\eps}
\end{align}
Since $\tht\in L^{\frac{2}{1-\zeta}}([0,\infty);\Hdot^{1+\be-2\al})$, we have that $\bar{T_\eps}<\infty$. From \eqref{T:epsilon:bar:def}, we can deduce that $([t_\eps,t_\eps+\bar{T_\eps}+1]\backslash B(\eps))\ne \phi;$. Therefore, there exists a time $\bar{t_\eps} \in ([t_\eps,t_\eps+\bar{T_\eps} +1]\backslash B(\eps))$ which satisfies $\Sob{\tht(\bar{t_\eps})}{\Hdot^{1+\be-2\al}}<\frac{\eps}{\sqrt{2}}$. Using \eqref{theta:T:epsilon}, we have
\begin{align*}
    \Sob{\tht(\bar{t_\eps})}{H^{1+\be-2\al}}^2&=\Sob{\tht(\bar{t_\eps})}{L^2}^2+\Sob{\tht(\bar{t_\eps})}{\Hdot^{1+\be-2\al}}^2\\
    &\le \Sob{\tht({t_\eps})}{L^2}^2+\Sob{\tht(\bar{t_\eps})}{\Hdot^{1+\be-2\al}}^2<\eps^2.
\end{align*}
From \eqref{energy:inequality:theta:T:epsilon}, we also have 
\begin{align*}
    \Sob{\bar{\tht_\eps}}{H^{1+\be-2\al}}\le \Sob{\tht_0}{H^{1+\be-2\al}}<(4C(\al,\be))^{-1}.
\end{align*}
Now we initialize \eqref{gSQG} with data ${\tht(\bar{t_\eps})}$. Then, by part ii), there exists the unique global solution of \eqref{gSQG} satisfying
\begin{align*}
    \Sob{\tht(t+\bar{t_\eps})}{H^{1+\be-2\al}}^2+2\int_{\bar{t_\eps}}^{t+\bar{t_\eps}}\Sob{\Lam^{\al}\tht(s)}{H^{1+\be-2\al}}^2\,ds\le \Sob{\tht(\bar{t_\eps})}{H^{1+\be-2\al}}^2, \quad \forall t\ge 0.
\end{align*}
In particular, we have
\[\Sob{\tht(t+\bar{t_\eps})}{H^{1+\be-2\al}}\le \Sob{\tht(\bar{t_\eps})}{H^{1+\be-2\al}}<\eps,\quad \forall t\ge 0.\]
This implies 
\[\Sob{\tht(s)}{H^{1+\be-2\al}}<\eps,\quad \forall s\ge \bar{t_\eps}.\]
In other words, 
\[\lim_{t\to \infty}\Sob{\tht(t)}{H^{1+\be-2\al}}=0,\]
thus completing the proof of Theorem 3.
\\
\appendix

\section{Proof of \cref{lem:commutator2}}\label{sec:app}
We give a sketch of the proof for \eqref{commutator2:A}. The proof for \eqref{commutator2:B} is similar.
First, let us define
    \begin{align*}
        \mathcal{L}^{s}_{i}(f,g,h):=\iint m_{s,i}(\xi,\eta)\hat{f}(\xi-\eta)\hat{g}(\eta)\overline{\hat{h}(\xi)}\, d\eta \, d\xi,
    \end{align*}
where
    \begin{align*}
        m_{s,i}(\xi,\eta):= \phi_{i}(\xi)\abs{\xi}^{s}\xi_{\ell}-\phi_{i}(\xe)\abs{\xi-\eta}^{s}(\xi-\eta)_{\ell}.
    \end{align*}
By Plancherel's theorem, we have
    \begin{align}\label{def:L:comm:relation}
        \mathcal{L}^{s}_{i}(f,g,h)=\lb[\Lam^{s}\bdy_{\ell}\Delta_i ,g]f,h\rb_{L^2}.
    \end{align}
It is therefore equivalent to obtain bounds on $\mathcal{L}_{i}^{s+\rho}$. 

Let $\mathbf{A}(\tau):=(\xi-\eta)+\tau \eta$. By the facts that $\supp \phi_{i} \subset \mathcal{A}_{i}$, $\supp \nabla \phi\subset\mathcal{A}_0$, and $\supp \hat{h}\in\mathcal{A}_j$, we have
 \begin{align}\label{E:meanvalue2}
        |m_{s+\rho,i}(\xi,\eta)|&=\bigg|\int_{0}^{1}\frac{d}{{d\tau}}\left(\phi_{i}(\mathbf{A}(\tau))\abs{\mathbf{A}(\tau)}^{s+\rho}\mathbf{A}(\tau)_{\ell}\right)d\tau\bigg|\notag\\
        &=\bigg|\int_{0}^{1}\bigg\{\nabla \phi(2^{-i}\mathbf{A}(\tau))\cdot(2^{-i}\eta)\mathbf{A}(\tau)_{\ell}\notag\\
        &\qquad\qquad+(s+\rho)\phi_{i}(\mathbf{A}(\tau))\abs{\mathbf{A}(\tau)}^{-2}\left(\mathbf{A}(\tau)\cdot \eta\right)\mathbf{A}(\tau)_{\ell}\notag\\
        & \qquad\qquad+\phi_{i}(\mathbf{A}(\tau))\eta_{\ell}\bigg\} \abs{\mathbf{A}(\tau)}^{s+\rho}d\tau\bigg|\notag\\
    &\le C\abs{\eta}\int_{0}^{1}\abs{\mathbf{A}(\tau)}^{s+\rho}d\tau\notag\\
    &\leq C_k\abs{\eta}\abs{\xi}^{s+\rho},
    \end{align}
 where we used the observation that $m_{s+\rho,i}(\xi,\eta)=0$ if $\abs{\mathbf{A}(\tau)}\notin \mathcal{A}_{i}$ and the assumption that $|i-j|\le k$.
We obtain 
\begin{align}
    |\mathcal{L}_{i}^{s+\rho}(f,g,h)|\le C\left|\iint |\xi|^{s} \hat{f}(\xi-\eta)\widehat{\Lam g}(\eta)\overline{\widehat{\Lam^{\rho}h}(\xi)}\, d\eta \, d\xi\right|
\end{align}
Applying lemma 4.1 in \cite{JollyKumarMartinez2020a} with $\si=s$ and $\eps=\nu$ and using Bernstein's inequality completes the proof.
\bibliographystyle{plain}
\bibliography{main_bib.bib}
\vspace{.3in}
\noindent Anuj Kumar\\ 
{\footnotesize
Department of Mathematics\\
Indian Institute of Technology Jodhpur, India\\
 Email: \url{akumar241@outlook.com}} \\[.2cm]
\end{document}